\documentclass[11pt]{amsart}
	
	\usepackage{amscd,amsmath,amsthm,amsfonts,amssymb}
	\usepackage{mathtools}
	\usepackage{geometry}
	\usepackage{enumitem}
	\usepackage{aliascnt}
	\usepackage[most]{tcolorbox}
	\usepackage[dvipsnames]{xcolor}
	
	\usepackage[
	colorlinks=true,
	linkcolor=blue,
	citecolor=cyan,
	urlcolor=NavyBlue
	]{hyperref}
	
	\allowdisplaybreaks
	
	\newcommand{\N}{\mathbb N}
	\newcommand{\R}{\mathbb R}
	
	\newcommand{\ND}{\mathcal{ND}}

	\newtheorem{theorem}{Theorem}[section]
	
	\newaliascnt{lemma}{theorem}
	\newtheorem{lemma}[lemma]{Lemma}
	\aliascntresetthe{lemma}

	\newaliascnt{corollary}{theorem}
	
	\aliascntresetthe{corollary}

	\newaliascnt{definition}{theorem}
	
	\aliascntresetthe{definition}

	\newaliascnt{remark}{theorem}
	
	\aliascntresetthe{remark}

	\newaliascnt{problem}{theorem}
	
	\aliascntresetthe{problem}

	\newtcolorbox{mainidea}[1]{
		enhanced,
		breakable,
		colback=black!2,
		colframe=black!55,
		boxrule=0.45pt,
		arc=1mm,
		left=7pt,
		right=7pt,
		top=6pt,
		bottom=6pt,
		title={#1},
		fonttitle=\bfseries
	}
	
	\title[\((n,\aleph_0)\)-Lineability of Nowhere Differentiable Functions]{
		On the \((n,\aleph_0)\)-Lineability of Nowhere Differentiable Functions
	}
	
	\author[F. Lima]{Fábio Lima}
	
	\address{
		Departamento de Matem\'atica\\
		Universidade Federal da Paraíba\\
		58.051-900 - Jo\~{a}o Pessoa, Brazil.}

	\email{fabiolimaoliveira99@gmail.com}
	
	\author[M. Santos]{Mónica Santos}
	
		\address{
		Faculdade de Matemática, Universidade Federal de Uberlândia, Uberlândia 38400-902, Brazil.}
	
	\email{monica.madeira@academico.ufpb.br}
	
	\author[G. Ribeiro]{Geivison Ribeiro}
	
	\address{
		Departamento de Matem\'atica\\
		Universidade Federal do Maranh\~ao\\
		65085--580 S\~ao Lu\'is, Brazil
	}
	
	\email{geivison.ribeiro@ufma.br}
	
	
	
\begin{document}
		
		\begin{abstract}
		We establish a countable-dimensional extension property for the family of continuous nowhere differentiable functions, which we denote by
		$\ND[0,1]$. More precisely, we prove that every finite-dimensional subspace
		$F\subset C[0,1]$ whose nonzero elements are nowhere differentiable can be
		extended to a countably infinite-dimensional subspace
		$G\subset C[0,1]$ with the same property.
		
		\end{abstract}
		
		\maketitle
		
		%================================================
		% INTRODUCTION
		%================================================
		
		\section{Introduction}
		
		The concept of \textit{lineability} was formalized in 2005 by Aron, Gurariy, and
		Seoane-Sepúlveda \cite{AGS}, although some of the first results in this
		direction appeared earlier. Roughly speaking, lineability theory is
		concerned with finding large vector subspaces contained, except possibly for
		the zero vector, in sets that are not themselves vector spaces. When the
		ambient space is endowed with a topology, one may also require such
		subspaces to be closed or dense, giving rise to the notions of spaceability
		and dense lineability, respectively. Since then, this research area has
		developed considerably and has found applications in several branches of
		Analysis. We refer the reader to the monograph \cite{AGPS} and the references
		therein for a broad account of the subject.
		
		A classical example in this context is the set of continuous nowhere
		differentiable functions. Throughout this paper, we write
		\[
		\ND[0,1]
		:=
		\left\{
		f\in C[0,1] : f \text{ is nowhere differentiable}
		\right\}.
		\]
		Although \(\ND[0,1]\) is not a vector space, it contains large linear
		structures. In particular, Fonf, Gurariy, and Kadets \cite{Gurariy} proved
		that there exists an infinite-dimensional subspace of \(C[0,1]\) whose
		nonzero elements are nowhere differentiable. Moreover, Rodríguez-Piazza
		\cite{Piazza} showed that every separable Banach space is isometric to a
		space of continuous nowhere differentiable functions. These results show
		that nowhere differentiability, despite being a highly irregular pointwise
		property, is compatible with rich linear structures.
		
		As the theory developed, stronger notions of lineability were introduced.
		Indeed, ordinary lineability guarantees only the existence of a large
		subspace contained in a given set. It does not say whether a subspace already
		contained in that set can be extended to a larger one while preserving the
		same property. To study this extension phenomenon, Fávaro, Pellegrino, and
		Tomaz introduced in \cite{FPT} the notions of
		\((\alpha,\beta)\)-lineability and \((\alpha,\beta)\)-spaceability. Other
		related refinements were later considered, such as pointwise lineability
		\cite{PR}, and general criteria for stronger forms of lineability were
		established in \cite{Raposo}.

	Let us recall the definitions used in this paper. Let \(V\) be a vector
	space, let \(A\subset V\), and let \(\alpha\) be a cardinal number. We say
	that \(A\) is \(\alpha\)-lineable if there exists an
	\(\alpha\)-dimensional subspace \(W\subset V\) such that $
	W\setminus\{0\}\subset A$. Equivalently, \(A\cup\{0\}\) contains an \(\alpha\)-dimensional vector subspace.
	
	Now let \(\alpha<\beta\) be cardinal numbers. The set \(A\) is said to be
	\((\alpha,\beta)\)-lineable if \(A\) is \(\alpha\)-lineable and every
	\(\alpha\)-dimensional subspace \(F\subset V\) satisfying $
	F\setminus\{0\}\subset A$
	can be extended to a \(\beta\)-dimensional subspace \(G\subset V\) such that
	\[
	F\subset G
	\qquad\text{and}\qquad
	G\setminus\{0\}\subset A.
	\]
	Thus, this notion concerns not only the existence of large linear
	structures. It also requires that every admissible subspace of dimension
	\(\alpha\) can be extended to a larger subspace of dimension \(\beta\)
	without leaving the set under consideration.

		Motivated by this extension problem, we investigate the
		\((n,\aleph_0)\)-lineability of \(\ND[0,1]\). More precisely, we ask whether
		every finite-dimensional subspace of \(C[0,1]\) whose nonzero elements are
		nowhere differentiable can be extended to a countably infinite-dimensional
		subspace with the same property.
		
		\medskip
		
		Notice that the classical lineability and spaceability results for
		\(\ND[0,1]\) do not answer this question. They ensure the existence of large
		subspaces contained in \(\ND[0,1]\cup\{0\}\), but those subspaces are chosen
		during the corresponding constructions. Here, the initial
		finite-dimensional subspace is arbitrary and prescribed in advance.
		Therefore, the new directions must be constructed in such a way that every
		nontrivial linear combination involving both the original subspace and the
		added functions remains nowhere differentiable.
		
		Our main result gives a positive answer to this question.
		
		\begin{theorem}\label{thm:main}
			For every \(n\in\N\), the set \(\ND[0,1]\) is
			\((n,\aleph_0)\)-lineable.
		\end{theorem}

		The proof is constructive. Starting from a prescribed finite-dimensional
		subspace \(F\subset C[0,1]\), we use the compactness of its closed unit ball
		to obtain a uniform modulus of continuity. We then construct continuous
		triangular waves with rapidly decreasing amplitudes and rapidly increasing
		frequencies. The parameters are chosen so that, along suitable increments
		tending to zero, one distinguished oscillatory term dominates both the
		contribution of \(F\) and the contributions of all remaining oscillatory
		terms. Finally, we group these functions into pairwise disjoint infinite
		blocks and obtain a countable linearly independent family adapted to \(F\).
		
		The paper is organized as follows. In \autoref{sec:preliminaries}, we introduce the notation,
		construct the oscillatory functions, and establish the estimates needed to
		control their difference quotients. In \autoref{sec:main-result}, we use these estimates to
		prove \autoref{thm:main}.
		
		%================================================
		% PRELIMINARIES
		%================================================
		
		\section{Preliminary results and notation}\label{sec:preliminaries}
		
		Throughout the paper, all functions are real-valued, and \(C[0,1]\) is
		endowed with the supremum norm. At the endpoints of \([0,1]\),
		differentiability is understood in the one-sided sense. Accordingly, if
		\(f\in C[0,1]\) and \(x\in[0,1]\), every difference quotient
		\[
		\frac{f(x+t)-f(x)}{t}
		\]
		is considered only for nonzero increments \(t\) such that
		\(x+t\in[0,1]\).
		
		We use the notation \(u_n=o(v_n)\), where \(v_n>0\), to mean that
		\(u_n/v_n\to0\) as \(n\to\infty\). Moreover, \(a_n\downarrow0\) means that
		\((a_n)\) is decreasing and converges to zero, while
		\(m_n\uparrow\infty\) means that \((m_n)\) is strictly increasing and
		unbounded. Whenever an estimate is said to hold uniformly in \(x\), the
		associated null sequence does not depend on \(x\).
		
		Let \(F\subset C[0,1]\) be a finite-dimensional subspace, and let
		\[
		B_F:=\{f\in F:\|f\|_\infty\le1\}.
		\]
		Since \(B_F\) is compact in \(C[0,1]\), it is equicontinuous. Hence the
		uniform modulus of continuity
		\[
		\omega_F(\delta)
		:=
		\sup\left\{
		|f(s)-f(t)|:
		f\in B_F,\ s,t\in[0,1],\ |s-t|\le\delta
		\right\}
		\]
		satisfies
		\begin{equation}\label{eq:modulus-continuity}
			\omega_F(\delta)\longrightarrow0
			\qquad\text{as }\delta\downarrow0.
		\end{equation}
		
		The next construction is the core of the proof. The functions \(g_n\) will
		have very small norms but very large difference quotients at the scale
		\(t_n(x)\). 
		
		\begin{mainidea}{Hierarchy of the relevant difference quotients}
			For suitable sequences and every \(x\in[0,1]\), the construction gives
			\[
			\left|
			\frac{g_n(x+t_n)-g_n(x)}{t_n}
			\right|
			=2a_nm_n,
			\]
			whereas
			\[
			\sup_{f\in B_F}
			\left|
			\frac{f(x+t_n)-f(x)}{t_n}
			\right|
			\le \frac{4}{n}a_nm_n,
			\qquad
			\sum_{i\ne n}
			\left|
			\frac{g_i(x+t_n)-g_i(x)}{t_n}
			\right|
			\le \varepsilon_n a_nm_n
			\]
			for a sequence \(\varepsilon_n\to0\).
		\end{mainidea}
		
		\begin{lemma}\label{lem:adapted-atoms}
			Let \(F\subset C[0,1]\) be finite-dimensional. Then there exist sequences
			\[
			(a_n)\subset(0,1),\qquad
			(m_n)\subset\N,\qquad
			(g_n)\subset C[0,1]
			\]
			such that
			\[
			a_n\downarrow0,\qquad
			\sum_{n=1}^\infty a_n<\infty,\qquad
			m_n\uparrow\infty,\qquad
			a_nm_n\longrightarrow\infty,
			\qquad
			\|g_n\|_\infty\le a_n.
			\]
			Moreover, for every \(x\in[0,1]\), there are nonzero numbers
			\(t_n=t_n(x)\) satisfying \(x+t_n\in[0,1]\),
			\(t_n\to0\), and
			\begin{enumerate}[label=\textnormal{(\roman*)}]
				\item
				\[
				\frac{1}{4m_n}\le |t_n|\le\frac{1}{2m_n};
				\]
				\item
				\[
				\left|
				\frac{g_n(x+t_n)-g_n(x)}{t_n}
				\right|
				=2a_nm_n;
				\]
				\item
				\[
				\sup_{f\in B_F}
				\left|
				\frac{f(x+t_n)-f(x)}{t_n}
				\right|
				\le \frac{4}{n}a_nm_n.
				\]
			\end{enumerate}
			The estimates are uniform in \(x\in[0,1]\).
		\end{lemma}
		
		\begin{proof}
			Set
			\[
			a_n:=2^{-2^n},
			\qquad n\in\N.
			\]
			Then \(a_n\downarrow0\) and \(\sum_n a_n<\infty\). We shall also use the
			strong tail estimate
			\begin{equation}\label{eq:tail-an}
				\frac{1}{a_n}\sum_{i>n}a_i\longrightarrow0.
			\end{equation}
			Indeed, since \(a_{n+j}=a_n^{2^j}\) and \(a_n\le1/4\), we have
			\[
			\sum_{i>n}a_i
			=
			\sum_{j=1}^\infty a_n^{2^j}
			\le
			\sum_{k=2}^\infty a_n^k
			=
			\frac{a_n^2}{1-a_n}
			\le
			\frac{4}{3}a_n^2,
			\]
			which proves \eqref{eq:tail-an}.
			
			We now choose \((m_n)\) recursively. Having chosen
			\(m_1,\ldots,m_{n-1}\), take an integer \(m_n>m_{n-1}\) sufficiently large
			so that
			\begin{equation}\label{eq:choice-frequency}
				a_nm_n\ge n,\qquad
				\omega_F\left(\frac{1}{2m_n}\right)\le\frac{a_n}{n},
				\qquad
				\sum_{i<n}a_im_i\le\frac{a_nm_n}{n}.
			\end{equation}
			This is possible by \eqref{eq:modulus-continuity}. In particular,
			\(m_n\uparrow\infty\) and \(a_nm_n\to\infty\).
			
			Consider the continuous \(1\)-periodic triangular wave
			\[
			\varphi(t):=2\,\operatorname{dist}(t,\mathbb Z),
			\qquad t\in\R,
			\]
			and, for \(N\in\N\), define
			\[
			\Psi_N(x):=\varphi(Nx),
			\qquad x\in[0,1].
			\]
			The function \(\Psi_N\) takes values in \([0,1]\), has supremum norm one,
			and is affine with slope \(2N\) or \(-2N\) on each interval
			\[
			J_{N,j}
			:=
			\left[\frac{j}{2N},\frac{j+1}{2N}\right],
			\qquad j=0,\ldots,2N-1.
			\]
			Set
			\[
			g_n:=a_n\Psi_{m_n}.
			\]
			Then \(\|g_n\|_\infty\le a_n\).
			
			Fix \(x\in[0,1]\). Choose one interval \(J_{m_n,j}\) containing \(x\);
			at a common endpoint of two adjacent intervals, either interval may be
			chosen. Let \(y_n(x)\) be an endpoint of this interval farthest from \(x\),
			and define
			\[
			t_n(x):=y_n(x)-x.
			\]
			Since the length of \(J_{m_n,j}\) is \(1/(2m_n)\), we obtain
			\[
			\frac{1}{4m_n}
			\le |t_n(x)|
			\le \frac{1}{2m_n}.
			\]
			In particular, \(t_n(x)\ne0\), \(x+t_n(x)\in[0,1]\), and \(t_n(x)\to0\).
			
			The points \(x\) and \(x+t_n(x)\) belong to the same affine piece of
			\(\Psi_{m_n}\). Therefore,
			\[
			\left|
			\frac{\Psi_{m_n}(x+t_n)-\Psi_{m_n}(x)}{t_n}
			\right|
			=2m_n,
			\]
			and hence
			\[
			\left|
			\frac{g_n(x+t_n)-g_n(x)}{t_n}
			\right|
			=2a_nm_n.
			\]
			
			Finally, for every \(f\in B_F\),
			\[
			|f(x+t_n)-f(x)|
			\le
			\omega_F(|t_n|)
			\le
			\omega_F\left(\frac{1}{2m_n}\right).
			\]
			Using \(|t_n|\ge1/(4m_n)\) and \eqref{eq:choice-frequency}, we get
			\[
			\left|
			\frac{f(x+t_n)-f(x)}{t_n}
			\right|
			\le
			4m_n\omega_F\left(\frac{1}{2m_n}\right)
			\le
			\frac{4}{n}a_nm_n.
			\]
			All the estimates are independent of \(x\), except for the choice of the
			increment \(t_n(x)\).
		\end{proof}
		
	The preceding lemma identifies one oscillatory term whose contribution is large. The next result shows that the combined contribution of all the other terms is negligible in comparison. This estimate will allow us to group the oscillatory functions while preserving the dominance of the selected term.

		\begin{lemma}\label{lem:off-diagonal}
			Under the assumptions and notation of \autoref{lem:adapted-atoms}, there is
			a sequence \((\varepsilon_n)\) of nonnegative numbers with
			\(\varepsilon_n\to0\) such that, for every \(x\in[0,1]\),
			\[
			\sum_{i\ne n}
			\left|
			\frac{g_i(x+t_n)-g_i(x)}{t_n}
			\right|
			\le
			\varepsilon_n a_nm_n.
			\]
			The estimate is uniform in \(x\).
		\end{lemma}
		
		\begin{proof}
			The function \(t\mapsto\operatorname{dist}(t,\mathbb Z)\) is
			\(1\)-Lipschitz. Hence \(\Psi_{m_i}\) is \(2m_i\)-Lipschitz, and \(g_i\)
			is \(2a_im_i\)-Lipschitz. Therefore, for \(i<n\),
			\[
			\left|
			\frac{g_i(x+t_n)-g_i(x)}{t_n}
			\right|
			\le 2a_im_i.
			\]
			By \eqref{eq:choice-frequency},
			\begin{equation}\label{eq:previous-atoms}
				\sum_{i<n}
				\left|
				\frac{g_i(x+t_n)-g_i(x)}{t_n}
				\right|
				\le
				2\sum_{i<n}a_im_i
				\le
				\frac{2}{n}a_nm_n.
			\end{equation}
			
			For \(i>n\), we use \(\|g_i\|_\infty\le a_i\) and
			\(|t_n|\ge1/(4m_n)\) to obtain
			\[
			\left|
			\frac{g_i(x+t_n)-g_i(x)}{t_n}
			\right|
			\le
			8a_im_n.
			\]
			Consequently,
			\begin{equation}\label{eq:future-atoms}
				\sum_{i>n}
				\left|
				\frac{g_i(x+t_n)-g_i(x)}{t_n}
				\right|
				\le
				8m_n\sum_{i>n}a_i
				=
				\left(
				\frac{8}{a_n}\sum_{i>n}a_i
				\right)a_nm_n.
			\end{equation}
			
			Set
			\[
			\varepsilon_n
			:=
			\frac{2}{n}
			+
			\frac{8}{a_n}\sum_{i>n}a_i.
			\]
			By \eqref{eq:tail-an}, \(\varepsilon_n\to0\). Combining
			\eqref{eq:previous-atoms} and \eqref{eq:future-atoms} completes the proof.
		\end{proof}
		
\medskip		
		
	We now divide the sequence of oscillatory functions into pairwise disjoint infinite sets. For each of these sets, we define a new function by summing the corresponding terms. Along a suitable subsequence, one of these functions provides the dominant contribution, while the contributions of all the remaining functions become negligible.

		\begin{lemma}\label{lem:block-functions}
			Let
			\[
			\N=\bigcup_{k=1}^\infty \N_k
			\]
			be a partition of \(\N\) into pairwise disjoint infinite subsets. Define
			\[
			h_k:=\sum_{n\in\N_k}g_n,
			\qquad k\in\N.
			\]
			Then \(h_k\in C[0,1]\) for every \(k\), and the sequence
			\((h_k)_{k=1}^\infty\) is linearly independent. Moreover, for every
			\(x\in[0,1]\), every \(k\in\N\), and every \(n\in\N_k\),
			\begin{equation}\label{eq:dominant-block}
				\left|
				\frac{h_k(x+t_n)-h_k(x)}{t_n}
				\right|
				\ge
				(2-\varepsilon_n)a_nm_n,
			\end{equation}
			whereas, for every \(j\ne k\),
			\begin{equation}\label{eq:other-block}
				\left|
				\frac{h_j(x+t_n)-h_j(x)}{t_n}
				\right|
				\le
				\varepsilon_n a_nm_n.
			\end{equation}
			In particular, the right-hand side of
			\eqref{eq:dominant-block} is at least \(a_nm_n\) for all sufficiently large
			\(n\in\N_k\).
		\end{lemma}
		
		\begin{proof}
			Since
			\[
			\sum_{n=1}^\infty\|g_n\|_\infty
			\le
			\sum_{n=1}^\infty a_n
			<
			\infty,
			\]
			each series defining \(h_k\) converges absolutely and uniformly on
			\([0,1]\). Hence \(h_k\in C[0,1]\).
			
			Fix \(x\in[0,1]\), \(k\in\N\), and \(n\in\N_k\). Since
			\[
			h_k
			=
			g_n+\sum_{\substack{i\in\N_k\\i\ne n}}g_i,
			\]
			\autoref{lem:adapted-atoms} and \autoref{lem:off-diagonal} yield
			\[
			\left|
			\frac{h_k(x+t_n)-h_k(x)}{t_n}
			\right|
			\ge
			2a_nm_n
			-
			\sum_{i\ne n}
			\left|
			\frac{g_i(x+t_n)-g_i(x)}{t_n}
			\right|,
			\]
			which proves \eqref{eq:dominant-block}. If \(j\ne k\), then
			\(n\notin\N_j\), and therefore
			\[
			\left|
			\frac{h_j(x+t_n)-h_j(x)}{t_n}
			\right|
			\le
			\sum_{i\ne n}
			\left|
			\frac{g_i(x+t_n)-g_i(x)}{t_n}
			\right|.
			\]
			This proves \eqref{eq:other-block}.
			
			It remains to prove linear independence. Suppose that
			\[
			\sum_{k=1}^p c_kh_k=0
			\]
			for some scalars \(c_1,\ldots,c_p\), and choose \(k_0\) such that
			\(c_{k_0}\ne0\). Fix \(x\in[0,1]\), and let \(n\to\infty\) through
			\(\N_{k_0}\). Taking difference quotients along \(t_n(x)\), and using
			\eqref{eq:dominant-block} and \eqref{eq:other-block}, we obtain
			\[
			|c_{k_0}|(2-\varepsilon_n)a_nm_n
			\le
			\varepsilon_n a_nm_n
			\sum_{\substack{1\le k\le p\\k\ne k_0}}|c_k|.
			\]
			After division by \(a_nm_n>0\) and passage to the limit, this gives
			\(2|c_{k_0}|\le0\), a contradiction. Thus \((h_k)\) is linearly independent.
		\end{proof}
		
		%================================================
		% MAIN RESULT
		%================================================
		
		\section{Proof of the main theorem}\label{sec:main-result}
		
		We are now ready to prove the extension property announced in the
		introduction.
		
		\begin{proof}[Proof of \autoref{thm:main}]
			Fix \(n\in\N\), and let \(F\subset C[0,1]\) be an \(n\)-dimensional
			subspace such that
			\[
			F\setminus\{0\}\subset\ND[0,1].
			\]
			Apply the construction of Section~\ref{sec:preliminaries} to \(F\), and let
			\((h_k)\) be the linearly independent sequence obtained in
			\autoref{lem:block-functions}. Set
			\[
			E:=\operatorname{span}\{h_k:k\in\N\},
			\qquad
			G:=F+E.
			\]
			
			We show that every nonzero element of \(G\) is nowhere differentiable. Let
			\[
			u=f+\sum_{k=1}^p c_kh_k,
			\qquad
			f\in F.
			\]
			If \(c_1=\cdots=c_p=0\), then \(u=f\), and the conclusion follows from the
			assumption on \(F\), provided \(u\ne0\).
			
			Suppose now that \(c_{k_0}\ne0\) for some
			\(k_0\in\{1,\ldots,p\}\). Fix \(x\in[0,1]\), and let \(n\to\infty\)
			through \(\N_{k_0}\). By \autoref{lem:adapted-atoms}, applied to
			\(f/\|f\|_\infty\) when \(f\ne0\),
			\[
			\left|
			\frac{f(x+t_n)-f(x)}{t_n}
			\right|
			\le
			\frac{4\|f\|_\infty}{n}a_nm_n.
			\]
			The same estimate is trivial when \(f=0\). By
			\autoref{lem:block-functions},
			\begin{align*}
				\left|
				\frac{u(x+t_n)-u(x)}{t_n}
				\right|
				&\ge
				|c_{k_0}|(2-\varepsilon_n)a_nm_n
				-
				\frac{4\|f\|_\infty}{n}a_nm_n
				\\
				&\quad
				-
				\varepsilon_n a_nm_n
				\sum_{\substack{1\le k\le p\\k\ne k_0}}|c_k|.
			\end{align*}
			Equivalently,
			\[
			\left|
			\frac{u(x+t_n)-u(x)}{t_n}
			\right|
			\ge
			\left[
			|c_{k_0}|(2-\varepsilon_n)
			-
			\frac{4\|f\|_\infty}{n}
			-
			\varepsilon_n
			\sum_{\substack{1\le k\le p\\k\ne k_0}}|c_k|
			\right]a_nm_n.
			\]
			The expression in brackets converges to \(2|c_{k_0}|>0\). Hence, for all
			sufficiently large \(n\in\N_{k_0}\),
			\[
			\left|
			\frac{u(x+t_n)-u(x)}{t_n}
			\right|
			\ge
			|c_{k_0}|a_nm_n.
			\]
			Since \(a_nm_n\ge n\), the right-hand side tends to infinity. Therefore the
			difference quotients of \(u\) at \(x\) are unbounded along a sequence of
			admissible increments tending to zero. Thus \(u\) is not differentiable at
			\(x\). Since \(x\in[0,1]\) was arbitrary, \(u\in\ND[0,1]\).
			
			The same argument shows that no nontrivial finite linear combination of the
			\(h_k\)'s can belong to \(F\). Consequently,
			\[
			G=F\oplus E.
			\]
			Since \(\dim F=n\) and \(\dim E=\aleph_0\), we have
			\(\dim G=\aleph_0\). Moreover,
			\[
			F\subset G
			\qquad\text{and}\qquad
			G\setminus\{0\}\subset\ND[0,1].
			\]
			Therefore \(\ND[0,1]\) is \((n,\aleph_0)\)-lineable.
		\end{proof}

		%%%%%%%%%%%%%%%%%%%%%%%%%%%%%%%%%%%%%%%%%%%%%%%%%%%%%%%%%%%%
		

\begin{thebibliography}{99}                                                                                               %
			
			
			\bibitem {ABRR}\href{https://doi.org/10.1007/s00574-023-00360-w}{G.
				Ara\'{u}jo, A. Barbosa, A. Raposo Jr., G. Ribeiro, \textit{On the spaceability
					of the set of functions in the Lebesgue space }$L_{p}$\textit{ which are not
					in }$L_{q}$, Bull. Braz. Math. Soc. $($N.S.$)$ \textbf{54} $($2023$)$, Paper
				no 44, 9pp.}
			
			\bibitem {Araujo}\href{https://doi.org/10.1007/s13398-023-01505-8}{G.
				Ara\'{u}jo and A. Barbosa, \textit{A general lineability criterion for
					complements of vector spaces}, Rev. Real Acad. Cienc. Exactas Fis. Nat. Ser.
				A-Mat. \textbf{118} $($2024$)$, no. 5.}
			
			\bibitem {AGPS}\href{https://doi.org/10.1201/b19277}{R.M. Aron, L.
				Bernal-Gonz\'{a}lez, D. Pellegrino, J.B. Seoane-Sep\'{u}lveda,
				\textit{Lineability: The Search for Linearity in Mathematics}. Monographs and
				Research Notes in Mathematics. CRC Press, Boca Raton $($2016$)$.}
			
			\bibitem {AGS}\href{https://doi.org/10.1090/S0002-9939-04-07533-1}{R.M. Aron,
				V.I. Gurariy, J.B. Seoane-Sep\'{u}lveda, \textit{Lineability and spaceability
					of sets of functions on }$\mathbb{R}$. Proc. Am. Math. Soc. \textbf{133}
				$($2005$)$, 795--803.}
			
			\bibitem {bernaljfa}{\href{https://doi.org/10.1016/j.jfa.2013.11.014}{{L.
						Bernal-Gonz\'{a}lez, M. Ord\'{o}\~{n}ez-Cabrera: Lineability criteria, with
						applications. J. Funct. Anal. \textbf{266} $($2014$)$, 3997-4025.}}}
			
			\bibitem {DR}\href{https://doi.org/10.1007/s00574-021-00246-9}{D. Diniz and A.
				Raposo Jr, \textit{A note on the geometry of certain classes of linear
					operators}, $($English summary$)$ Bull. Braz. Math. Soc. $($N.S.$)$
				\textbf{52} $($2021$)$, 1073--1080.}
			
			\bibitem {FPT}\href{https://doi.org/10.1007/s00574-019-00142-3}{V.V.
				F\'{a}varo, D. Pellegrino, D. Tomaz, \textit{Lineability and spaceability: a
					new approach}, Bull. Braz. Math. Soc. \textbf{51} $($2020$)$, 27--46.}
			
			\bibitem {Raposo}\href{https://doi.org/10.1090/proc/16608}{V.V. F\'{a}varo, D.
				Pellegrino, A.B. Raposo J\'{u}nior, G.S. Ribeiro, \textit{General criteria for
					a stronger notion of lineability}, Proc. Amer. Math. Soc. \textbf{152} $($3$)$
				$($2024$)$, 941--954.}
			
			
			\bibitem {Gurariy}\href{https://mathscinet.ams.org/mathscinet/relay-station?mr=1738120}{V. Fonf, V. Gurariy, M. Kadets, \textit{An infinite dimensional subspace of $C[0,1]$ consisting of nowhere differentiable functions}. C. R. Acad. Bulgare Sci. \textbf{52} (1999), no. 11-12, 13–16.}
			
			\bibitem {Piazza}\href{https://doi.org/10.1090/S0002-9939-1995-1328375-8}{L.
				Rodr\'{\i}guez-Piazza, \textit{Every Separable Banach space is isometric to a
					space of continuous nowhere differentiable functions}. Proc. Amer. Math. Soc.
				\textbf{123} $($1995$)$, no. 12, 3649--3654.}
			
			\bibitem {Leo}\href{https://doi.org/10.1112/blms.12858}{P. Leonetti, T. Russo,
				J. Somaglia, \textit{Dense lineability and spaceability in certain subsets of}
				$\ell_{\infty}$, Bull. Lond. Math. Soc. \textbf{55} $($2023$)$, 1--21.}
			
			\bibitem {PR}\href{https://doi.org/10.1016/j.indag.2020.12.006}{D. Pellegrino
				and A. Raposo Jr, \textit{Pointwise lineability in sequence spaces}. Indag.
				Math. $($N.S.$)$ \textbf{32} $($2021$)$, 536--546.}
			
			
		\end{thebibliography}
		\end{document}